\documentclass[a4paper,12pt]				{amsart}
\usepackage[T1]								{fontenc} 							
\usepackage[utf8]						    {inputenc} 							
\usepackage									{ellipsis} 							
\usepackage									{lmodern} 							
\usepackage									{amssymb, amsmath, amsthm}			
\usepackage[foot]                           {amsaddr}                           
\usepackage									{bm, bbm}							
\usepackage[fixamsmath]						{mathtools} 						
\usepackage[shortlabels]					{enumitem} 							
\usepackage									{csquotes}							
\usepackage[spacing,kerning]				{microtype}
\usepackage									{geometry}
\usepackage 								{hyperref}
\usepackage[style=alphabetic,backend=biber]	{biblatex}
\usepackage									{todonotes,color}

\microtypecontext{spacing=nonfrench}

\theoremstyle{plain}
\newtheorem{theorem}					{Theorem}[section]
\newtheorem{lemma}			[theorem]	{Lemma}
\newtheorem{corollary}		[theorem]	{Corollary}
\newtheorem{proposition}	[theorem]	{Proposition}

\theoremstyle{definition}
\newtheorem{definition}		[theorem]	{Definition}

\theoremstyle{remark}
\newtheorem*{exa*}					{Example}
\newtheorem*{rem*}					{Remark}

\DeclareMathOperator	{\IR}		{\mathbb{R}}
\DeclareMathOperator	{\IE}		{\mathbb{E}} 
\DeclareMathOperator	{\IP}		{\mathbb{P}} 
\DeclareMathOperator	{\IN}		{\mathbb{N}} 

\DeclarePairedDelimiter	{\abs}		{\lvert}	{\rvert}
\DeclarePairedDelimiter	{\norm}		{\lVert}	{\rVert}

\newcommand				{\eins}		{\text{$\mathbbm{1}$}}
\newcommand		    	{\dpartial}	{\mathfrak{d}}

\renewcommand 			{\epsilon}	{\varepsilon}
\renewcommand			{\phi}		{\varphi}
\renewcommand			{\tilde}	{\widetilde}

\numberwithin			{equation}{section}

\begin{document}
\title{Concentration inequalities on the multislice and for sampling without replacement}
\thanks{This research was supported by the German Research Foundation (DFG) via CRC 1283 ``\textit{Taming uncertainty and profiting from randomness and low regularity in analysis, stochastics and their applications}''.}
\author{Holger Sambale$^1$ and Arthur Sinulis$^1$}
\address{$^1$Faculty of Mathematics, Bielefeld University, Bielefeld, Germany}
\email{\{hsambale, asinulis\}@math.uni-bielefeld.de}
\keywords{concentration of measure, convex distance inequality, Erd\H{o}s--R\'{e}nyi graphs, multislice, sampling without replacement}

\begin{abstract}
We present concentration inequalities on the multislice which are based on (modified) log-Sobolev inequalities. This includes bounds for convex functions and multilinear polynomials. As an application we show concentration results for the triangle count in the $G(n,M)$ Erd\H{o}s--R\'{e}nyi model resembling known bounds in the $G(n,p)$ case. Moreover, we give a proof of Talagrand's convex distance inequality for the multislice.

Interpreting the multislice in a sampling without replacement context, we furthermore present concentration results for $n$ out of $N$ sampling without replacement. Based on a bounded difference inequality involving the finite-sampling correction factor $1- n/N$, we present an easy proof of Serfling's inequality with a slightly worse factor in the exponent, as well as a sub-Gaussian right tail for the Kolmogorov distance between the empirical measure and the true distribution of the sample.
\end{abstract}
\maketitle

\section{Introduction}\label{Section:Introduction}
In the past few years, in particular in the analysis of Boolean functions, a model which has found emerging interest is the \emph{multislice}. It can be regarded as a natural generalization of several well-known models like slices of the hypercube. In detail, let $L \ge 2$ be a natural number, $\kappa = (\kappa_1, \ldots, \kappa_L) \in \mathbb{N}^L$ (where by convention, $0 \notin \mathbb{N}$), $N \coloneqq \kappa_1 + \cdots + \kappa_L$, and let $\mathcal{X} = \{x_1, \ldots, x_L\} \subset \IR$ be a set of $L$ distinct real numbers. The multislice is defined as
\[
    \Omega_\kappa \coloneqq \Big\{\omega = (\omega_1, \ldots, \omega_N) \in \mathcal{X}^N \colon \sum_{i=1}^N \eins_{\{\omega_i = x_\ell\}} = \kappa_\ell\, \text{ for } \ell = 1, \ldots, L\Big\}.
\]
In other words, any $\omega \in \Omega_\kappa$ is a sequence of elements from $\{x_1, \ldots, x_L\}$ in which each feature $x_\ell$ appears exactly $\kappa_\ell$ times. In the context of sampling without replacement, it describes the procedure of (fully) sampling from a population with a set of characteristics $\{x_1, \ldots, x_L\}$, such that a proportion of $\kappa_\ell / N$ of the population has characteristic $x_\ell$. We discuss and extend this relation in Section \ref{sect:swor}.

To gain an intuition into the multislice let us consider some special choices of $L$ and $\kappa$. For $L=2$ and $\kappa = (k,N-k)$, the multislice reduces to $k$-slices on the hypercube, while the case of $L=N$ and $\kappa = (1, \ldots, 1)$ can be interpreted as the symmetric group $S_N$. If $L=2$, $\Omega_\kappa$ can be interpreted as all possible realizations of an Erd\H{o}s--R\'{e}nyi random graph (see Proposition \ref{prop:ERDr} below for more details).
Moreover, the multislice gives rise to a Markov chain known as the \emph{multi-urn Bernoulli--Laplace diffusion model}, but we will not pursue this aspect. For examples, see \cite{Sal20}.

We equip $\Omega_\kappa$ with the uniform distribution which we denote by $\IP_\kappa$ or sometimes also simply $\IP$. In other words,
\[
\IP_\kappa(\{\omega\}) = \abs{\Omega_\kappa}^{-1} = \binom{N}{\kappa_1, \ldots, \kappa_L}^{-1}
\]
for any $\omega \in \Omega_\kappa$. If $f$ is any real-valued function on $\Omega_\kappa$, we write $\mathbb{E}_\kappa f$ or $\mathbb{E}f$ for its expectation with respect to $\IP_\kappa$. Moreover, to fix some conventions, we always assume the $x_\ell$ to be ordered such that $x_1 < x_2 < \ldots < x_L$. In particular, $\abs{\mathcal{X}} \coloneqq x_L - x_1$ denotes the diameter of $\mathcal{X}$. Furthermore, we shall write $\kappa_{\mathrm{min}} \coloneqq \min\{\kappa_1, \ldots, \kappa_L\}$. Finally, for any $1 \le i \ne j \le N$, let $\tau_{ij}$ be the ``switch'' operator which switches the $i$-th and $j$-th component of the vector $\omega$. In other words, $\tau_{ij}$ transforms $\omega$ into the vector $\tau_{ij}\omega$ given by
\begin{equation}\label{switchop}
\tau_{ij}\omega = (\omega_1, \ldots, \omega_{i-1}, \omega_j, \omega_{i+1}, \ldots, \omega_{j-1}, \omega_i, \omega_{j+1}, \ldots, \omega_N).
\end{equation}

Multislices equipped with the uniform measure were also considered in earlier works. Logarithmic Sobolev inequalities were proven in \cite{FOW19, Sal20}, while in \cite{Fil20}, the Friedgut--Kalai--Naor (FKN) theorem was extended to the multislice. We shall make use of the functional inequalities proven by Salez \cite{Sal20} to apply the entropy method and prove concentration inequalities in the above-mentioned settings.

\subsection{Concentration inequalities for various types of functionals}\label{sect:CISpF}
In the first section, we present concentration inequalities for some functions on the multislice which are comparable to known concentration results in the independent case. We begin with a number of elementary inequalities. 

\begin{proposition}\label{elementaryresults}
\begin{enumerate}
    \item Let $f \colon \Omega_\kappa \to \IR$ be a function such that $\abs{f(\omega) - f(\tau_{ij}\omega)} \le c_{ij}$ for all $\omega \in \Omega_\kappa$, all $1 \le i < j \le N$ and suitable constants $c_{ij} \ge 0$. For any $t \ge 0$, we have
    \begin{equation}\label{bddiffms}
    \IP_\kappa (f - \IE_\kappa f \ge t) \le \exp\Big(- \frac{Nt^2}{4\sum_{1 \le i < j \le N}c_{ij}^2}\Big).
    \end{equation}
    \item Let $f \colon [x_1, x_L]^N \to \IR$ be convex and $1$-Lipschitz. Then, for any $t \ge 0$ we have
    \begin{equation}\label{conv+1Lip}
    \IP_\kappa (f-\IE_\kappa f \ge t) \le \exp\Big(-\frac{t^2}{16\abs{\mathcal{X}}^2}\Big).
    \end{equation}
\end{enumerate}
\end{proposition}

Proposition \ref{elementaryresults} follows by a classic approach of Ledoux \cite{Led97} (the entropy method), i.\,e.\ by exploiting suitable log-Sobolev-type inequalities, some of which might be of independent interest (cf.\ Propositions \ref{prop:LSImultislice} and \ref{convexnablaLSI}). Note that the bounded differences-type inequality \eqref{bddiffms} is invariant under the change $f \mapsto -f$, so that in particular, this result extends to the concentration inequality
\begin{equation}\label{bddiffci}
\IP_\kappa (\abs{f - \IE_\kappa f} \ge t) \le 2\exp\Big(- \frac{Nt^2}{4\sum_{1 \le i < j \le N}c_{ij}^2}\Big).
\end{equation}
By contrast, \eqref{conv+1Lip} clearly does not hold for $-f$ in general, but by different techniques discussed in Section \ref{sect:TCD}, this result can be extended to the lower tails as well.

While results for Lipschitz-type functions as in Proposition \ref{elementaryresults} are fairly standard in concentration of measure theory, in the past decade there has been increasing interest in non-Lipschitz functions. A case in point are so-called \emph{multilinear polynomials}, i.\,e.\ polynomials which are affine with respect to every variable. Clearly, any multilinear polynomial $f = f(\omega)$ of degree $d$ may be written as
\begin{equation}\label{homogmlp}
f(\omega) = a_0 + \sum_{i_1=1}^N a_{i_1}\omega_{i_1} + \cdots + \sum_{i_1 < \ldots < i_d} a_{i_1 \ldots i_d} \omega_{i_1} \cdots \omega_{i_d}.
\end{equation}

Typically, multilinear polynomials of degree $d \ge 2$ no longer have sub-Gaussian tails, but the tails show different regimes or levels of decay, corresponding to a larger family of norms of the tensors of derivatives $\nabla^k f$, $k = 1, \ldots, d$. For $t$ large, terms of the form $\exp(-(t/\beta_d)^{2/d})$ dominate, where $\beta_d$ depends on the $d$-th order derivatives. Tail inequalities of this type are also called \emph{multilevel tail inequalities}, a term phrased by Adamczak \cite{ABW17, AKPS18}.

In detail, we need a family of norms $\norm{\cdot}_{\mathcal{I}}$ on the space of $d$-tensors for each partition $\mathcal{I} = \{ I_1, \ldots, I_k \} \in P_d$, where $P_d$ denotes the set of all partitions of $\{1,\ldots,d\}$. For any $1 \le i_1, \ldots, i_d \le N$ and any subset $I \subset \{1, \ldots, d\}$, we write $i_I = (i_k)_{k \in I}$, and for each $\ell = 1,\ldots, k$ we denote by $x^{(\ell)}$ a vector in $\IR^{N^{I_\ell}}$. Then, for a $d$-tensor $A = (a_{i_1, \ldots, i_d})$ and a partition $\mathcal{I} \in P_d$, we set
\[
\norm{A}_{\mathcal{I}} \coloneqq \sup \Big\lbrace \sum_{i_1, \ldots, i_d} a_{i_1 \ldots i_d} \prod_{\ell = 1}^k x^{(\ell)}_{i_{I_\ell}} : \sum_{i_{I_\ell}} (x^{(\ell)}_{i_{I_\ell}})^2 \le 1 \text{ for all } \ell = 1, \ldots, k\Big\rbrace.
\]

The family $\norm{\cdot}_{\mathcal{I}}$ was first introduced in \cite{La06}, where it was used to prove two-sided estimates for $L^p$ norms of Gaussian chaos, and the definitions given above agree with the ones from \cite{La06} as well as \cite{AW15} and \cite{AKPS18}. We can regard the $\norm{A}_{\mathcal{I}}$ as a family of operator-type norms. In particular, it is easy to see that $\norm{A}_{\{1, \ldots, d\}} = \norm{A}_\mathrm{HS} \coloneqq (\sum_{i_1, \ldots, i_d} a_{i_1 \ldots i_d}^2)^{1/2}$ (Hilbert--Schmidt norm) and $\norm{A}_{\{\{1\}, \ldots, \{d\}\}} = \norm{A}_\mathrm{op} \coloneqq \sup \{ \sum_{i_1, \ldots, i_d} a_{i_1 \ldots i_d} x^{(1)}_{i_1} \cdots x^{(d)}_{i_d} : \abs{x_{i^{(\ell)}}} \le 1 \text{ for all } \ell = 1, \ldots, d \}$ (operator norm).

\begin{theorem}\label{multlinpol}
Let $f = f(\omega)$ be a multilinear polynomial \eqref{homogmlp} of degree $d$. There exists a constant $c = c(d)$ such that
\[
\IP_\kappa (\abs{f - \IE_\kappa f} \ge t) \le 2 \exp\Big(-c \min_{1\le k \le d} \min_{\mathcal{I} \in P_k}\Big(\frac{t}{\abs{\mathcal{X}}^k\norm{\IE_\kappa \nabla^k f}_\mathcal{I}}\Big)^{2/\abs{\mathcal{I}}}\Big).
\]
\end{theorem}

Theorem \ref{multlinpol} is an analogue of \cite[Theorem 1.4]{AW15} (independent sub-Gaussian random variables), \cite[Theorem 2.2]{AKPS18} (the Ising model), \cite[Theorem 5]{GSS18b} (in presence of certain discrete log-Sobolev inequalities) and \cite[Corollary 5.4]{APS20} (modified log-Sobolev inequalities for Glauber dynamics) for the multislice.

For the sake of illustration, consider the case of $d=2$ and a quadratic form $f(\omega) = \sum_{i<j} a_{ij}\omega_i\omega_j = \omega^TA\omega/2$, where $A$ is a symmetric matrix with vanishing diagonal and entries $A_{ij} = a_{ij} = A_{ji}$ for any $i < j$. Let us additionally assume that $\IE_\kappa \omega_i = 0$ for any $i$. In this case, we obviously have $\IE_\kappa \nabla f = 0$ and $\IE_\kappa \nabla^2 f = A$. Consequently, the conclusion of Theorem \ref{multlinpol} reads
\[
\IP_\kappa (\abs{f - \IE_\kappa f} \ge t) \le 2 \exp\Big(-c \min \Big(\frac{t^2}{\abs{\mathcal{X}}^4\norm{A}_\mathrm{HS}^2}, \frac{t}{\abs{\mathcal{X}}^2\norm{A}_\mathrm{op}}\Big)\Big),
\]
showing a version of the famous Hanson--Wright inequality for the multislice (cf.\ \cite{HW71}). As an alternate strategy of proof, in Section \ref{sect:TCD} we derive Talagrand's convex distance inequality for the multislice, which in particular yields Hanson--Wright inequalities by \cite{AW15} (where results of this type have already been established for sampling without replacement.) Theorem \ref{multlinpol} may be seen as a generalization of these bounds to any order $d \in \mathbb{N}$.

Possible applications include the Erd\H{o}s--R\'{e}nyi model, which features random graphs with a fixed number of vertices $n$. There are two variants of the Erd\H{o}s--R\'{e}nyi model which are often labeled $G(n,p)$ and $G(n,M)$. In the $G(n,p)$ model, each possible edge between the $n$ vertices is included with probability $p$ independently of the other edges, while in the $G(n,M)$ model, the graph is chosen uniformly at random from the collection of all graphs with $n$ vertices and $M$ edges. In the following we study $G(n,M)$.

Write $E = \{(i,j): 1 \le i < j \le n\}$ for the set of possible edges, so that $\mathrm{card}(E) = n(n-1)/2 \eqqcolon N$. Clearly, any edge $e \in E$ is included with probability $M/N \eqqcolon p$. However, unlike in the $G(n,p)$ model, the edges are not independent. Any configuration $\omega$ in $G(n,M)$ can be written as a vector $\omega = (\omega_e)_{e\in E} \in \{0,1\}^E$ such that $\omega_e = 1$ for exactly $M$ entries. In particular, $G(n,M)$ can be regarded as a multislice with $L=2$, $\kappa = (N-M,M)$ and $\mathcal{X} = \{0,1\}$.

One problem which has attracted considerable attention over the last two decades is the number of copies of certain subgraphs, e.\,g.\ triangles, in the Erd\H{o}s--R\'{e}nyi model. There is extensive literature on concentration inequalities for the triangle count, such as \cite{JR02}, \cite{Cha12} and \cite{DK12a}. In particular, in \cite[Proposition 5.5]{AW15}, bounds for the $G(n,p)$ model are derived using higher order concentration results for multilinear polynomials in independent random variables. As Theorem \ref{multlinpol} provides analogous higher order concentration results in a dependent situation, we are able to show corresponding bounds for the $G(n,M)$ model by our methods.

\begin{proposition}\label{prop:ERDr}
    Consider the $G(n,M)$ Erd\H{o}s--R\'{e}nyi model, and let $f(\omega) \coloneqq \sum_{i<j<k} \omega_{ij}\omega_{jk}\omega_{ik}$ be the number of triangles. Then, for any $t \ge 0$,
    \[
    \IP(\abs{f - \IE f} \ge t) \le 2 \exp\Big(-c \min \Big(\frac{t^2}{n^3 + p^2 n^3 + p^4 n^4}, \frac{t}{n^{1/2} + p n}, t^{2/3}\Big)\Big).
    \]
\end{proposition}

Comparing Proposition \ref{prop:ERDr} to \cite[Proposition 5.5]{AW15}, we see that we arrive at essentially the same tail bounds despite the dependencies in the $G(n,M)$ model, with the only difference of an additional logarithmic factor $L_p \coloneqq (\log(2/p))^{-1/2}$ in \cite{AW15}. This logarithmic factor stems from the use of sub-Gaussian norms for independent Bernoulli random variables (which tend to $0$ as $p \to 0$), which is not mirrored in the log-Sobolev tools we use.

Typically, the main interest is to study fluctuations which scale with the expected value of $f$. In this case, setting $t \coloneqq \varepsilon \IE f = \varepsilon \binom{n}{3} M(M-1)(M-2)/(N(N-1)(N-2))$, Proposition \ref{prop:ERDr} reads
    \[
    \IP(\abs{f - \IE f} \ge \varepsilon \IE f) \le 2 \exp\big(-c \min\big(\varepsilon^2 n^3 p^6, (\varepsilon^2 \wedge \varepsilon^{2/3})n^2 p^2 \big)\big).
    \]
In particular, this shows that the optimal exponent $n^2p^2$ known from the $G(n,p)$ setting also shows up for a suitable range of $p$, cf.\ the discussion in \cite{AW15}.

In a similar way, we may also count cycles as in \cite[Proposition 5.6]{AW15}, but we do not pursue this in this note.

\subsection{Sampling without replacement}\label{sect:swor}
In this section we interpret the multislice in the sampling without replacement context, where we sample $N$ times from a population of $N$ individuals $\omega_1, \ldots, \omega_N$, so that the uniform distribution $\IP_\kappa$ describes the sampling of all its elements. In applications one does not sample the entire population, but chooses some sample size $n \le N$, i.\,e.\ for each $\omega \in \Omega_\kappa$, and considers the first $n$ coordinates only. Formally, if $pr_n$ denotes the projection onto the first $n$ coordinates, we may define $\Omega_{\kappa,n} \coloneqq pr_n(\Omega_\kappa)$. We, again, equip $\Omega_{\kappa,n}$ with the uniform distribution $\IP_{\kappa,n}$, which agrees with the push-forward of $\IP_\kappa$ under $pr_n$. As above, we denote the expectation with respect to $\IP_{\kappa,n}$ by $\IE_{\kappa,n}f$, where $f$ is any real-valued function.

Our first result is a bounded differences inequality for sampling without replacement involving the finite-sampling correction factor $1- n/N$. In the sequel, $(\omega_{i^c}, \omega_i')$ denotes a vector which agrees with $\omega$ in all coordinates but the $i$-th one, while $\omega_i$ is replaced by some admissible $\omega_i'$ (in the sense that $(\omega_{i^c}, \omega_i') \in \Omega_{\kappa,n}$). Moreover, for any $\sigma \in S_n$ we may define $\sigma \omega \in \Omega_{\kappa,n}$ by noting that $\sigma$ acts on $\omega$ by permuting its indices.

\begin{proposition} \label{theorem:subgtswor}
    Let $f: \Omega_{\kappa,n} \to \IR$ be an arbitrary function and $(c_i)_{i =1,\ldots,n}$ such that $\abs{f(\omega) - f(\omega_{i^c}, \omega_i')} \le c_i$ for all $\omega \in \Omega_{\kappa,n}, \omega_i' \in \mathcal{X}$. For any $t \ge 0$ it holds
    \begin{align} \label{eqn:bdasymmetric}
        \IP_{\kappa,n}\Big(\frac{1}{n!} \sum_{\sigma \in S_n} f(\sigma \omega) - \IE_{\kappa,n} f \ge t \Big) \le \exp\Big( - \frac{t^2}{4(1-\frac{n}{N}) \sum_{i = 1}^n c_i^2} \Big).
    \end{align}
    In particular, if $f$ is symmetric and satisfies $\abs{f(\omega) - f(\omega_1', \omega_2, \ldots, \omega_n)} \le c$ for some $c > 0$, this implies
    \begin{align}\label{eqn:swrboundeddiff1}
        \IP_{\kappa,n}\big( f - \IE_{\kappa,n} f \ge t \big) \le \exp\Big( - \frac{t^2}{4\big( 1 - \frac{n}{N} \big) c^2 n} \Big).
    \end{align}
\end{proposition}

Note that equation \eqref{eqn:bdasymmetric} is invariant under the change $f \mapsto -f$, which yields a two-sided concentration inequality as in \eqref{bddiffci}.

To express it in terms of deviation probabilities, for any $\delta \in (0,1]$ we have with probability at least $1-\delta$
\[
    \Big \lvert \frac{1}{n!} \sum_{\sigma \in S_n} f(\sigma \omega) - \IE_{\kappa,n} f \Big \rvert \le \sqrt{4(1-n/N) \log\big( \frac{2}{\delta} \big) \sum_{i =1}^n c_i^2 }.
\]
Concentration inequalities of this type have also been proven in \cite[Lemma 2]{EP09} and \cite[Theorem 5]{CMPR09} by different methods, and our results agree with these bounds up to constants.

Let us apply Proposition \ref{theorem:subgtswor} to some known statistics in sampling without replacement. One of the most famous concentration results for sampling without replacement is Serfling's inequality \cite{Se74}, which can be regarded as a strengthening of Hoeffding's inequality for $n$ out of $N$ sampling due to the inclusion of the finite-sampling correction factor $1-n/N$. For a discussion and some newer results we refer to \cite{BM15}, \cite{To17} and \cite{GW17}. We can deduce Serfling's inequality with a slightly worse constant from Proposition \ref{theorem:subgtswor}.

\begin{corollary}\label{corollary:Serfling}
    In the situation above,  we have for any $t \ge 0$
    \[
        \IP_{\kappa,n}\Big( \frac{1}{n} \sum_{i = 1}^n \omega_i - \IE_{\kappa,n} \omega_1 \ge t \Big) \le \exp\Big( - \frac{n t^2}{4\big(1- \frac n N \big)\abs{\mathcal{X}}^2} \Big).
    \]
    The same estimate holds for $\IP_{\kappa,n}( \frac{1}{n} \sum_{i = 1}^n \omega_i - \IE_{\kappa,n} \omega_1 \le -t)$.
\end{corollary}

In the original version of Serfling's inequality, the right-hand side is replaced by $\exp(-2nt^2/((1-(n-1)/N)\abs{\mathcal{X}}^2))$. 

As a second example, consider the approximation of the the uniform distribution on all the points from which the $\omega_i$ are sampled using the empirical measure, measured in terms of the Kolmogorov distance. Formally, we put
\[
    g_{n,t}(\omega_1, \ldots, \omega_n) = \frac{1}{n} \sum_{i = 1}^n \eins_{(-\infty, t]}(\omega_i)
\]
and
\[
    f(\omega) \coloneqq \sup_{t \in \IR} (g_{n,t}(\omega) - \IE_{\kappa,n} g_{n,t}).
\]
In \cite{GW17}, it was conjectured that $\sqrt{n} f$ has sub-Gaussian tails with variance $1 - n/N$. The next result states that after centering around the expectation, this is indeed the case. 

\begin{corollary}\label{corollary:subgconj}
With the above notation we have for any $t \ge 0$
\[
\IP_{\kappa,n}( \sqrt{n}\abs{f - \IE_{\kappa,n} f} \ge t) \le 2\exp\Big( - \frac{t^2}{4\big( 1 - \frac{n}{N} \big)} \Big).
\]
\end{corollary}

\subsection{Talagrand's convex distance inequality}\label{sect:TCD}
Let $\Omega$ be any measurable space, $\omega = (\omega_1, \ldots, \omega_N) \in \Omega^N$ and $A \subset \Omega^N$ a measurable set. In his landmark paper \cite{Tal95}, Talagrand defined the convex distance between $\omega$ and $A$
\[
    d_T(\omega, A) \coloneqq \sup_{\alpha \in \mathbb{R}^N: \abs{\alpha} = 1} d_\alpha(\omega, A),
\]
where
\[
    d_\alpha(\omega, A) \coloneqq \inf_{\omega' \in A} d_\alpha(\omega, \omega') \coloneqq \inf_{\omega' \in A} \sum_{i=1}^N \abs{\alpha_i} \eins_{\omega_i \neq \omega'_i}.
\]

Talagrand proved concentration inequalities for the convex distance of random permutations and product measures which have attracted continuous interest since then. For product measures, an alternate proof based on the entropy method was given in \cite{BLM09}. In \cite{SS19}, the entropy method was used to reprove the convex distance inequality for random permutations as well, and this proof was extended to slices of the hypercube. In the present article, we further generalize this proof to the multislice, encompassing both situations discussed in \cite{SS19}.

\begin{proposition}\label{TcdeMS}
	For any $A \subseteq \Omega_\kappa$ it holds
	\[
	\IP_\kappa(A) \IE_\kappa \exp\Big( \frac{d_T(\cdot, A)^2}{144} \Big) \le 1.
	\]
\end{proposition}

Note that in \cite{Pau14}, convex distance inequalities for certain types of dependent random variables are proven. This includes sampling without replacement. In this sense, the result of Proposition \ref{TcdeMS} is not new, but we present a different strategy of proof solely based on the entropy method.

A famous corollary of Talagrand's convex distance inequality are sub-Gaussian concentration inequalities for convex Lipschitz functions, as first proven in \cite{Tal88}. Thus, Proposition \ref{TcdeMS} implies the following corollary, which can be regarded as an extension of Proposition \ref{elementaryresults} to upper \emph{and} lower tails (ignoring the subtle issue of concentration around the mean or the median of a function).

\begin{corollary}\label{cor1}
Let $f: \IR^N \to \IR$ be convex and $L$-Lipschitz. Then for any $t \ge 0$ it holds
\[
\IP_\kappa\big( \abs{f - \mathrm{med}(f)} \ge t \big)\le 4 \exp\Big( - \frac{t^2}{144L^2 \abs{\mathcal{X}}^2} \Big),
\]
where $\mathrm{med}(f)$ is a median for $f$.
\end{corollary}

As a simple application of Corollary \ref{cor1}, we show the following bound on the largest eigenvalue of symmetric matrices whose entries have distribution $\IP_\kappa$:

\begin{corollary}\label{cor:largestEV}
Let $X = (X_{ij})_{i,j}$ be a symmetric $n \times n$ random matrix. Let $N \coloneqq n(n+1)/2$ and assume that the common distribution of the entries $(X_{ij})_{i \le j}$ on $\IR^N$ is given by $\IP_\kappa$ for some $\kappa$, $L \ge 2$ and $\mathcal{X}$. Let $\lambda_\mathrm{max} \coloneqq \lambda_\mathrm{max}(X) \coloneqq \max \{\abs{\lambda(X)} \colon \lambda(X) \ \text{eigenvalue of} \ X\}$. We have for any $t \ge 0$
\[
    \IP(\abs{\lambda_\mathrm{max}(X) - \mathrm{med}(\lambda_\mathrm{max}(X))} \ge t) \le 4 \exp\Big( - \frac{t^2}{144 \abs{\mathcal{X}}^2} \Big).
\]
\end{corollary}

In particular, this result shows that $\lambda_\mathrm{max}$ has sub-Gaussian tails independently of the dimension $n$. A possible choice of $X$ is the adjacency matrix of a $G(n,M)$ Erd\H{o}s--R\'{e}nyi random graph. Corollary \ref{cor:largestEV} is an adaption of a classical example for independent random variables, see e.\,g.\ \cite[Example 6.8]{BLM13}. 

Furthermore, we are able to prove a somewhat weaker version of the convex distance inequality for $n$ out of $N$ sampling. Here we consider symmetric sets, i.\,e.\ sets $A \subset \Omega_{\kappa,n}$ such that $\omega \in A$ implies $\sigma\omega \in A$ for any permutation $\sigma \in S_n$. Obviously, assuming $A$ to be symmetric is increasingly restrictive if $n$ tends to $N$. This is mirrored in the additional finite-sampling correction factor $1-n/N$ in the following theorem (which sharpens the convex distance inequality in \cite{Pau14}).

\begin{theorem}\label{thm:convswor}
For any symmetric set $A \in \Omega_{\kappa,n}$ with $\IP_{\kappa,n}(A) \ge \frac{1}{2}$ and any $t \ge 0$ we have
\[
\IP_{\kappa,n}\big( d_T(\cdot, A) \ge t \big) \le e \exp\Big( - \frac{t^2}{16(1-\frac n N)} \Big).
\]
\end{theorem}

As above, Theorem \ref{thm:convswor} implies the following result.

\begin{corollary} \label{corollary:convSymmLip}
Let $f$ be a convex and symmetric $L$-Lipschitz function. Then for any $t \ge 0$ we have
\[
    \IP_{\kappa,n}\big( \abs{f - \mathrm{med} (f)} \ge t \big) \le 2e \exp\Big( - \frac{t^2}{16(1-n/N) L^2 \abs{\mathcal{X}}^2} \Big).
\]
\end{corollary}

Examples of functions to which Corollary \ref{corollary:convSymmLip} may be applied are the estimators for the mean and the standard deviation given by $f(\omega) = \bar{\omega} = n^{-1} \sum_{i =1}^n \omega_i$ (sample mean) and $f(\omega) = s(\omega) = (\frac{1}{n-1} \sum_{i=1}^n (\omega_i - \bar{\omega})^2)^{1/2} = (\frac{1}{n(n-1)} \sum_{i < j} (\omega_i - \omega_j)^2)^{1/2}$ (sample standard deviation), having Lipschitz constants $L = n^{-1/2}$ and $L=(2n)^{-1/2}$, respectively. In particular, for any $\delta \in (0,1]$ we have with probability at least $1-\delta$ for any of the two estimators
\[
    \abs{f - \mathrm{med}(f)} \le \sqrt{16 (1-n/N) L^2 \abs{\mathcal{X}}^2 \log(2e/\delta)}.
\]

It is well-known that concentration results centered around the expectation and the median differ only by a constant. Indeed, in our case, for any convex, symmetric $L$-Lipschitz function
\begin{align*}
    \abs{\IE_{\kappa, n} f - \mathrm{med}(f)} &\le \IE_{\kappa,n} \abs{f - \mathrm{med}(f)} = \int_0^\infty \IP_{\kappa,n}\big( \abs{f - \mathrm{med}(f)} \ge t \big) dt \\
    &\le 2e \int_0^\infty \exp\Big( - \frac{t^2}{16(1-n/N) \mathcal{X}^2 L^2} \Big) dt = 2e \sqrt{4\pi(1-n/N) \abs{\mathcal{X}}^2 L^2} \\
    &\approx 19.27 \abs{\mathcal{X}} L \sqrt{1-n/N}.
\end{align*}

\section{Logarithmic Sobolev inequalities for the multislice}\label{sect:LSIs}
The main tool for establishing concentration inequalities in this note is the entropy method, which is based on the use of logarithmic Sobolev-type inequalities. Let us recall some basic facts and definitions especially adapted to discrete spaces. A key object is a suitable difference operator, i.\,e.\ a kind of ``discrete derivative''. Given a probability space $(\mathcal{Y}, \mathcal{F}, \mu)$, we call any operator $\Gamma: L^\infty(\mu) \to L^\infty(\mu)$ satisfying $|\Gamma (af + b)| = a\, |\Gamma f|$ for all $a > 0$, $b \in \mathbb{R}$ a difference operator. Moreover, by $\IE_\mu$ we denote integration with respect to $\mu$.

\begin{definition}
\begin{enumerate}
    \item We say that $\mu$ satisfies a logarithmic Sobolev inequality $\Gamma\mathrm{-LSI}(\sigma^2)$ if for all bounded measurable functions $f$, we have
\[
	\mathrm{Ent}_{\mu}(f^2) \le 2\sigma^2 \IE_\mu \Gamma(f)^2,
\]
where $\mathrm{Ent}_\mu(f) \coloneqq \IE_\mu f\log(f) - \IE_\mu f \log\IE_\mu f$ (for any positive function $f$) denotes the entropy functional.
\item We say that $\mu$ satifies a modified logarithmic Sobolev inequality $\Gamma\mathrm{-mLSI}(\sigma^2)$ if for all bounded measurable functions $f$, we have
\[
	\mathrm{Ent}_{\mu}(e^f) \le \frac{\sigma^2}{2} \IE_\mu \Gamma(f)^2e^f.
\]
\item We say that $\mu$ satisfies a Poincar\'{e} inequality $\Gamma\mathrm{-PI}(\sigma^2)$ if for all bounded measurable functions $f$, we have
\[
	\mathrm{Var}_{\mu}(f) \le \sigma^2 \IE_\mu \Gamma(f)^2,
\]
where $\mathrm{Var}(f) \coloneqq \IE_\mu f^2 - (\IE_\mu f)^2$ is the variance.
\item If any of these functional inequalities does not hold for \emph{all} bounded measurable functions but for some subclass $\mathcal{A} \subset L^\infty(\mu)$, we say that $\mu$ satisfies a $\Gamma\mathrm{-LSI}(\sigma^2)$ (PI, mLSI) \emph{on $\mathcal{A}$}.
\end{enumerate}
\end{definition}
If $\Gamma$ satisfies the chain rule (as the ordinary gradient $\nabla$ does), $\Gamma\mathrm{-LSIs}$ and $\Gamma\mathrm{-mLSIs}$ are equivalent concepts, but in the examples we consider in this note, this is usually not true. Moreover, it is well-known that a $\Gamma\mathrm{-LSI}(\sigma^2)$ implies a $\Gamma\mathrm{-PI}(\sigma^2)$, cf.\ e.\,g.\ \cite[Proposition 3.6]{BT06}.

For the multislice, we mostly consider the following canonical difference operator. Recalling the ``switch'' operator from \eqref{switchop}, for any function $f \colon \Omega_\kappa \to \mathbb{R}$ we set
\[
\Gamma_{ij}(f)(\omega) \coloneqq \Gamma_{ij}f(\omega) \coloneqq f(\omega) - f(\tau_{ij}\omega) \eqqcolon f(\omega) - \tau_{ij}f(\omega)
\]
and define the difference operator $\Gamma$ by
\[
\Gamma (f) \coloneqq \Big(\frac{1}{2N} \sum_{1 \le i < j \le N} \Gamma_{ij} (f)^2\Big)^{1/2}.
\]

Note that $\Gamma_{ij}(f)^2$ might be interpreted as a sort of ``local variance''. Indeed, it is easy to verify that
\begin{equation}\label{alt&neu}
\Gamma_{ij}(f)^2(\omega) = 2 \int (f(\omega) - f(\omega_{\{i,j\}^c}, \eta_{ij}))^2 d\IP_\kappa(\eta_{ij} \mid \omega_{\{i,j\}^c}),
\end{equation}
where $\omega_{\{i,j\}^c} = (\omega_k)_{k \notin \{i,j\}}$ and $\eta_{ij} = (\eta_i, \eta_j)$. Therefore, we have $\Gamma (f)^2 = 2 N^{-1} \abs{\mathfrak{d}f}^2$ for the difference operator $\abs{\dpartial f}$ introduced in \cite{GSS18}.

Sometimes (and typically for auxiliary purposes), we shall also need a second, closely related difference operator which we denote by $\Gamma^+$. Here, we simply set
\[
\Gamma_{ij}^+(f)(\omega) \coloneqq (f(\omega) - f(\tau_{ij}\omega))_+,
\]
where $x_+ \coloneqq \max(x,0)$ denotes the positive part of a real number, and define $\Gamma^+$ accordingly.

Recently, in \cite{Sal20} sharp (modified) logarithmic Sobolev inequalities for the multislice were established. Rewriting these results in accordance with our notation and slightly extending them immediately leads to the following proposition, serving as the basis for our arguments:

\begin{proposition}\label{prop:LSImultislice}
With the above definitions of $\Gamma$ and $\Gamma^+$, $\IP_\kappa$ satisfies the following functional inequalities:
\begin{itemize}
    \item $\Gamma\mathrm{-LSI}(2\log(N/\kappa_{\mathrm{min}})/\log(2))$,
    \item $\Gamma\mathrm{-mLSI}(4)$,
    \item $\Gamma^+\mathrm{-mLSI}(8)$.
\end{itemize}
\end{proposition}

\begin{proof}[Proof of Proposition \ref{prop:LSImultislice}]
    The $\Gamma\mathrm{-LSI}$ directly follows from \cite[Theorem 5]{Sal20}. Moreover, by \cite[Lemma 1]{Sal20} (substituting $f \ge 0$ by $e^f$), we have
    \begin{equation}\label{LSInachSal}
    \mathrm{Ent}_{\IP_\kappa}(e^f) \le \frac{1}{N} \sum_{i < j} \IE_\kappa(e^{f(\tau_{ij}\omega)} - e^{f(\omega)})(f(\tau_{ij}\omega) - f(\omega))
    \end{equation}
    for any $f \colon \Omega_\kappa \to \IR$. Using the fact that $\omega \mapsto \tau_{ij}\omega$ is an automorphism of $\Omega_\kappa$ and applying the inequality $(a-b)(e^a - e^b) \le \frac{1}{2} (e^a + e^b)(a-b)^2$ leads to the $\Gamma\mathrm{-mLSI}(4)$. By similar arguments, we may also deduce the $\Gamma^+\mathrm{-mLSI}(8)$. In particular, we note that the expected values on the right-hand side of \eqref{LSInachSal} are symmetric in $\omega$ and $\tau_{ij}\omega$ and use the inequality $(a-b)_+ (e^a - e^b) \le (a-b)_+^2 e^a$.
\end{proof}

From Proposition \ref{prop:LSImultislice} we may derive a convex $\nabla-\mathrm{(m)LSI}$ on the multislice, where $\nabla$ denotes the usual Euclidean gradient.

\begin{proposition}\label{convexnablaLSI}
For any $f \in \mathcal{A}_c \coloneqq \{ f \colon [x_1,x_L]^N \to \IR \mid f \ \text{convex}\}$, we have
\[
\mathrm{Ent}_{\IP_\kappa}(e^f) \le 4\abs{\mathcal{X}}^2\IE_\kappa e^f \abs{\nabla f}^2.
\]
In other words, $\IP_\kappa$ satisfies a $\nabla-\mathrm{mLSI}(8\abs{\mathcal{X}}^2)$ on $\mathcal{A}_c$.
\end{proposition}

\begin{proof}
Using convexity in the first step and the inequality $(a-b)^2 \le 2a^2 + 2b^2$ yields
    \begin{align*}
        \Gamma^+(f)^2(\omega) &= \frac{1}{4N} \sum_{i \neq j} (f(\omega) - f(\tau_{ij}\omega))_+^2 \le \frac{1}{4N} \sum_{i \neq j} \langle\nabla f(\omega), \omega - \tau_{ij}\omega\rangle^2 \\
        &= \frac{1}{4N} \sum_{i \neq  j} (\omega_i-\omega_j)^2\big( \partial_i f(\omega) - \partial_j f(\omega) \big)^2 \\
        &\le \frac{\abs{\mathcal{X}}^2}{2N} \sum_{i \neq j} \big( \partial_i f(\omega)^2 + \partial_j f(\omega)^2 \big) \\
        &= \frac{\abs{\mathcal{X}}^2}{N} (N-1) \sum_{i = 1}^N \partial_i f(\omega)^2
        \le \abs{\mathcal{X}}^2 \abs{\nabla f}^2(\omega).
    \end{align*}
    As $\IP_\kappa$ satisfies a $\Gamma^+-\mathrm{mLSI}(8)$ by Proposition \ref{prop:LSImultislice}, the claim follows.
\end{proof}

Another class of functional inequalities we address in this note are Beckner inequalities. Restricting ourselves to the multislice (rather than providing a general definition), $\IP_\kappa$ satisfies a Beckner inequality with parameter $p \in (1,2]$ (Bec-$p$) if there exists some constant $\beta_p > 0$ such that
\begin{equation}\label{Beckner}
    \IE_\kappa f^p - (\IE_\kappa f)^p \le \frac{\beta_p p}{2} \mathcal{E}_\kappa(f, f^{p-1})
\end{equation}
for any nonnegative function $f$. Here,
\[
\mathcal{E}_\kappa(f,g) \coloneqq \frac{1}{2N} \sum_{1 \le i < j \le N} \IE_\kappa (\Gamma_{ij} f)(\Gamma_{ij} g)
\]
for any functions $f, g$ on $\Omega_\kappa$ (which is the Dirichlet form of the underlying Markov chain).

Recently, in \cite{APS20} it was shown that in the context of general Markov semigroups, Beckner inequalities with constants bounded away from zero as $p \downarrow 1$ and modified log-Sobolev inequalities are equivalent. In their article, the authors provide numerous examples and applications, also briefly discussing the multislice. Since we need results of this type for our purposes, we include a somewhat more detailed discussion in the present note.

\begin{proposition}\label{Becknerms}
    For any $p \in (1,2]$, $\IP_\kappa$ satisfies a Beckner inequality Bec-$p$ with constant $\beta_p = \frac{4N}{p(N+2)}$.
\end{proposition}

\begin{proof}
First note that the result holds true for $\kappa = (1, \ldots, 1)$ as proven in \cite[Proposition 4.8]{BT06}, with the difference in the constant being due to different normalizations. To extend this result to general $\kappa$, we apply a ``projection'' or ``coarsening'' argument, cf.\ \cite[Section 3.4]{Sal20}. Indeed, consider the map $\Psi \colon \{1, \ldots, N\} \to \{1, \ldots, L\}$ given by $\Psi(i) = \ell$ iff $i \in \{\kappa_1 + \cdots + \kappa_{\ell-1} + 1, \ldots, \kappa_1 + \cdots + \kappa_\ell\}$ and extend it to the multislice by coordinate-wise application, i.\,e.\ $\Psi (\omega_1, \ldots, \omega_N) \coloneqq (\Psi(\omega_1), \ldots, \Psi(\omega_N))$. Then, by \cite[Lemma 4]{Sal20},
\[
\IE_\kappa f = \IE_{(1, \ldots, 1)}(f \circ \Psi),\qquad \mathcal{E}_\kappa(f,g) = \mathcal{E}_{(1, \ldots, 1)}(f \circ \Psi, g \circ \Psi)
\]
for any functions $f,g$. From these identities, we immediately obtain the result.
\end{proof}

Finally, we may also derive logarithmic Sobolev inequalities for symmetric functions of sampling without replacement. Here we use other types of difference operators. Let $f \colon \Omega_{\kappa,n} \to \IR$ be any (not necessarily symmetric) function. Then, we set
\begin{align*}
\mathfrak{h}(f)^2(\omega_1, \ldots, \omega_n) &= \frac{1}{2} \sum_{i = 1}^n (\sup_{\omega_i} f(\omega_1, \ldots, \omega_n) - \inf_{\omega_i'} f(\omega_1, \ldots, \omega_{i-1}, \omega_i', \omega_{i+1}, \ldots, \omega_n))^2,\\
\mathfrak{h}^+(f)^2(\omega_1, \ldots, \omega_n) &= \frac{1}{2} \sum_{i = 1}^n (f(\omega_1, \ldots, \omega_n) - \inf_{\omega_i'} f(\omega_1, \ldots, \omega_{i-1}, \omega_i', \omega_{i+1}, \ldots, \omega_n))^2.
\end{align*}
Here, the supremum and the infimum have to be interpreted as extending over all admissible configurations, i.\,e.\ such that $(\omega_{i^c},\omega_i), (\omega_{i^c},\omega_i') \in \Omega_{\kappa,n}$.

\begin{proposition}\label{prop:LSIswor}
Let $\mathcal{A}_{n,s} \coloneqq \{f \colon \Omega_{\kappa,n} \to \IR \mid f \ \text{symmetric}\}$. With the above definitions of $\mathfrak{h}$ and $\mathfrak{h}^+$, $\IP_{\kappa,n}$ satisfies the following functional inequalities on $\mathcal{A}_{n,s}$:
\begin{itemize}
    \item $\mathfrak{h}\mathrm{-LSI}(2\log(N/\kappa_{\mathrm{min}})(1-\frac{n}{N})/\log(2))$,
    \item $\mathfrak{h}\mathrm{-mLSI}(4(1-\frac{n}{N}))$,
    \item $\mathfrak{h}^+\mathrm{-mLSI}(8(1-\frac{n}{N}))$.
\end{itemize}
\end{proposition}

\begin{proof}
We only prove the $\mathfrak{h}^+\mathrm{-mLSI}$. The proofs of the other two inequalities follow by a modification of the arguments below.

First note that any function $f$ on $\Omega_{\kappa,n}$ can be extended to a function $F$ on $\Omega_\kappa$ which only depends on the first $n$ coordinates by setting $F(\omega_1, \ldots, \omega_N) \coloneqq f(\omega_1, \ldots, \omega_n)$, which may be rewritten as $F = f \circ pr_n$. We now apply Proposition \ref{prop:LSImultislice} to $F$. Obviously, $\mathrm{Ent}_{\IP_\kappa}(e^F) = \mathrm{Ent}_{\IP_{\kappa,n}}(e^f)$. It therefore remains to consider the right-hand side of the $\mathrm{mLSI}$. Here we obtain
\begin{align*}
    &\frac{1}{2N} \sum_{i <j} \IE_\kappa (F(\omega) - F(\tau_{ij} \omega))_+^2 e^{F(\omega)}
    = \frac{1}{2N} \sum_{i = 1}^n \sum_{j = n+1}^N \IE_\kappa (F(\omega) - F(\tau_{ij}\omega))_+^2 e^{F(\omega)} \\
    \le \ &\frac{1}{2N} \sum_{i = 1}^n \sum_{j = n+1}^N \IE_\kappa (F(\omega) - \inf_{\omega_i'} F(\omega_{i^c}, \omega_i'))_+^2  e^{F(\omega)} \\
    = \ &\frac{N-n}{2N} \sum_{i = 1}^n \IE_{\kappa,n} (f(\omega_1, \ldots, \omega_n) - \inf_{x_i'} f(\omega_1, \ldots, \omega_{i-1}, x_i', \omega_{i+1}, \ldots, \omega_n))_+^2 e^{f(\omega_1, \ldots, \omega_n)}.
\end{align*}
Here, the first equality follows by symmetry of $f$ with respect to the symmetric group $S_n$, and the fact that $f$ does not depend on $(\omega_{n+1}, \ldots, \omega_N)$. The first inequality is due to the monotonicity of $x \mapsto x_+$, and the last equality follows as $\IP_{\kappa,n}$ is the push-forward of $\IP_\kappa$ under $pr_n$. Thus, for any $f \in \mathcal{A}_{n,s}$ it holds
\begin{align*}
\mathrm{Ent}_{\IP_{\kappa,n}}(e^f) &\le \frac{4}{2N} \sum_{i < j} \IE_\kappa (F(\omega) - F(\tau_{ij} \omega))_+^2 e^{F(\omega)}\\ &\le 4 \frac{N-n}{N} \IE_{\kappa,n} \mathfrak{h}^+(f)^2(\omega_1, \ldots, \omega_n) e^{f(\omega_1, \ldots, \omega_n)},
\end{align*}
which finishes the proof.
\end{proof}

\section{Proofs of the concentration inequalities}\label{sect:Proofs1}

\begin{proof}[Proof of Proposition \ref{elementaryresults}]
Recall that if a probability measure $\mu$ satisfies a $\Gamma-\mathrm{mLSI}(\sigma^2)$ on $\mathcal{A}$ (where $\Gamma$ denotes some difference operator), we have for any $f \in \mathcal{A}$ such that $\Gamma(f) \le L$,
    \begin{equation}\label{eqn:SS19-1.2}
    \mu(f-\IE f \ge t) \le \exp\Big(-\frac{t^2}{2 \sigma^2 L^2}\Big)
    \end{equation}
for any $t \ge 0$. For a reference, see e.\,g.\ \cite{BG99} or \cite[(1.2)]{SS19}. Combining this fact with Proposition \ref{prop:LSImultislice} and noting that by definition,
\[
\Gamma(f) \le \Big(\frac{1}{2N} \sum_{1 \le i < j \le N} c_{ij}^2\Big)^{1/2},
\]
we arrive at \eqref{bddiffms}. In the same way, we may derive \eqref{conv+1Lip} using Proposition \ref{convexnablaLSI}.
\end{proof}

The proof of Theorem \ref{multlinpol} is more advanced. The basic idea is to follow the steps of the proof of \cite[Theorem 2.2]{AKPS18} and its refinements as presented in \cite[Section 5.3]{APS20}. First, we derive moment estimates for functions on the multislice.

\begin{lemma}\label{newmomest}
For any $f \colon \Omega_\kappa \to \IR$ and any $p \ge 2$,
\[
\norm{f - \IE_\kappa f}_{L^p(\IP_\kappa)} \le \sqrt{4\theta p} \norm{\Gamma(f)}_{L^p(\IP_\kappa)},
\]
where $\theta \coloneqq \sqrt{e}/(\sqrt{e}-1) < 2.5415$.
\end{lemma}

\begin{proof}
This follows immediately from Proposition \ref{Becknerms} and \cite[Proposition 3.3]{APS20}. To apply the latter result, we have to check that the the constants of the Beckner inequalities Bec-$p$ satisfy
\[
\beta_p^{-1} = \frac{p(N+2)}{4N} \ge a (p-1)^s
\]
for some $a > 0$, $s \ge 0$ and any $p \in (1,2]$. Clearly, we may take $a=1/4$ and $s=0$, which finishes the proof.
\end{proof}

Note that alternatively, we could apply \cite[Proposition 2.4]{GSS18}, using \eqref{alt&neu} and Proposition \ref{prop:LSImultislice}, which yields
\[
\lVert f - \IE_\kappa f \rVert_{L^p(\IP_\kappa)} \le \sqrt{\frac{8\log(N/\kappa_{\mathrm{min}})}{\log(2)}} \sqrt{p-1} \lVert \Gamma(f) \rVert_{L^p(\IP_\kappa)}.
\]
As a result of using the $\Gamma-\mathrm{LSI}$, we arrive at a substantially weaker constant, however.

Next, we have to relate differences of multilinear polynomials to (formal) derivatives, which is typically achived by an inequality of the form $\Gamma(f) \le c \abs{\nabla f}$ for some absolute constant $c > 0$. However, it comes out that such an inequality cannot be true in our setting. For instance, taking $N=3$, $\mathcal{X} = \{0,1\}$ and $f(\omega) = \omega_1\omega_2 - \omega_1\omega_3$, it is easy to check that for $\omega = (0,1,1)$, we have $0 = \abs{\nabla f(\omega)} < \Gamma (f)(\omega)$. The same problem arises if we take $\Gamma^+$ instead of $\Gamma$. It is possible to prove an inequality of this type with $c \coloneqq \abs{\mathcal{X}}$ for multilinear polynomials with non-negative coefficients and $\mathcal{X} \subset [0, \infty)$ (this can be seen by slightly modifying the proof of Proposition \ref{prop:do+der} below). However, the proof of Theorem \ref{multlinpol} also includes an iteration and linearization procedure, and if we only allow for non-negative coefficients we get stuck at $d=2$.

The following proposition provides us with the estimate we need to get the recursion going, at the cost of also involving second order derivatives.

\begin{proposition}\label{prop:do+der}
    Let $f = f(\omega)$ be a multilinear polynomial as in Theorem \ref{multlinpol}. Then we have
    \begin{equation}\label{eqn:gradest}
    \Gamma(f)^2 \le \frac{3\abs{\mathcal{X}}^2}{2} \abs{\nabla f}^2 + \frac{3\abs{\mathcal{X}}^4}{4N} \norm{\nabla^2 f}_\mathrm{HS}^2.
    \end{equation}
    In particular, for any $p \ge 2$ we have
    \begin{equation}\label{eqn:momest}
        \norm{f - \IE_{\kappa} f}_{L^p(\IP_\kappa)}
        \le \ \sqrt{6 \theta \abs{\mathcal{X}}^2 p} \norm{ \abs{\nabla f}}_{L^p(\IP_\kappa)} + \sqrt{3 \theta \abs{\mathcal{X}}^4 p/N} \norm{ \norm{\nabla^2 f}_\mathrm{HS}}_{L^p(\IP_\kappa)}
    \end{equation}
    with $\theta$ as in Lemma \ref{newmomest}.
\end{proposition}

\begin{proof}
    In the proof, we additionally assume $f$ to be $d$-homogeneous, i.\,e.\
    \[
    f(\omega) = \sum_{i_1 < \ldots < i_d} a_{i_1 \ldots i_d} \omega_{i_1} \cdots \omega_{i_d}.
    \]
    This is done in order to ease notation, and it is no problem to extend our proof to the non-homogeneous case. For notational convenience, for any $i_1 < \ldots < i_d$ and any permutation $\sigma \in S_d$, we define $a_{i_{\sigma(1)} \ldots i_{\sigma(d)}} \coloneqq a_{i_1 \ldots i_d}$, and we set $a_{i_1 \ldots i_d} = 0$ if $i_j = i_k$ for some $j \neq k$. Finally, note that some of the notation below has to be interpreted accordingly for small values of $d$, e.\,g.\ summation over $i_1 < \ldots < i_{d-1}$ reduces to summation over $i_1$ for $d=2$.
    Observe that for any $k, \ell \in \{1,\ldots,N\}, k \neq \ell,$ we have
    \begin{align*}
        &\Gamma_{k\ell}(f)(\omega)^2\\
        =\ &\big( \sum_{\substack{i_1 < \ldots < i_{d-1} \\ \ell \notin \{i_1, \ldots, i_{d-1} \}}} a_{i_1 \ldots i_{d-1} k} \omega_{i_1} \cdots \omega_{i_{d-1}} (\omega_k - \omega_\ell)\\
        &\hspace{2cm}+ \sum_{\substack{i_1 < \ldots < i_{d-1} \\ k \notin \{i_1, \ldots, i_{d-1} \}}} a_{i_1 \ldots i_{d-1} \ell} \omega_{i_1} \cdots \omega_{i_{d-1}} (\omega_\ell - \omega_k) \big)^2 \displaybreak[2]\\
        =\ &\big( \sum_{i_1 < \ldots < i_{d-1}} a_{i_1 \ldots i_{d-1} k} \omega_{i_1} \cdots \omega_{i_{d-1}} (\omega_k - \omega_\ell)
        + \sum_{i_1 < \ldots < i_{d-1}} a_{i_1 \ldots i_{d-1} \ell} \omega_{i_1} \cdots \omega_{i_{d-1}} (\omega_\ell - \omega_k)\\
        &\hspace{2cm}+ \sum_{i_1 < \ldots < i_{d-2}} a_{i_1 \ldots i_{d-2} k\ell} \omega_{i_1} \cdots \omega_{i_{d-2}} (\omega_k - \omega_\ell)^2\big)^2\displaybreak[2]\\
        \le\ &3\abs{\mathcal{X}}^2 \Big(\big( \sum_{i_1 < \ldots < i_{d-1}} a_{i_1 \ldots i_{d-1} k} \omega_{i_1} \cdots \omega_{i_{d-1}} \big)^2
        + \big( \sum_{i_1 < \ldots < i_{d-1}} a_{i_1 \ldots i_{d-1} \ell} \omega_{i_1} \cdots \omega_{i_{d-1}} \big)^2\Big) \\
        &\hspace{2cm}+ 3 \abs{\mathcal{X}}^4 \big(\sum_{i_1 < \ldots < i_{d-2}} a_{i_1 \ldots i_{d-2} k\ell} \omega_{i_1} \cdots \omega_{i_{d-2}}\big)^2\\
        =\ &3\abs{\mathcal{X}}^2 (\partial_k f(\omega)^2 + \partial_\ell f(\omega)^2) + 3\abs{\mathcal{X}}^4 \partial_{k\ell} f(\omega)^2.
    \end{align*}
    Consequently it holds
    \begin{align*}
        \Gamma(f)^2 &= \frac{1}{4N} \sum_{k \neq \ell} \Gamma_{k\ell}(f)^2 \le \frac{3\abs{\mathcal{X}}^2}{4N} \sum_{k\neq \ell} ((\partial_k f)^2 + (\partial_\ell f)^2) + \frac{3\abs{\mathcal{X}}^4}{4N} \sum_{k\neq \ell} (\partial_{kl} f)^2\\
        &\le \frac{3\abs{\mathcal{X}}^2}{2} \abs{\nabla f}^2 + \frac{3\abs{\mathcal{X}}^4}{4N} \norm{\nabla^2 f}_\mathrm{HS}^2,
    \end{align*}
    proving equation \eqref{eqn:gradest}. Finally, combining \eqref{eqn:gradest} with Lemma \ref{newmomest}, we immediately arrive at \eqref{eqn:momest}.
\end{proof}

With the help of Proposition \ref{prop:do+der}, we may now prove Theorem \ref{multlinpol}. To this end, let us introduce some additional notation.
If $A = (a_{i_1 \ldots i_k})_{i_1, \ldots, i_k \le N}$, $B = (b_{i_1 \ldots i_k})_{i_1, \ldots, i_k \le N}$ are two $k$-tensors, we define an inner product $\langle \cdot, \cdot \rangle$ by
\[
\langle A, B \rangle \coloneqq \sum_{i_1, \ldots, i_k \le N} a_{i_1 \ldots i_k} b_{i_1 \ldots i_k}.
\]
Moreover, if $x^j = (x^j_1, \ldots, x^j_N)$, $j = 1, \ldots, k$, are any vectors, we set $x^1 \otimes \cdots \otimes x^k \coloneqq (x^1_{i_1} \cdots x^k_{i_k})_{i_1, \ldots, i_k \le N}$. We also extend this notation to the situation in which some of these vectors may be $N^2$-dimensional. Indeed, let $x^1, \ldots, x^k$ be $N$-dimensional vectors as above, and let $y^1, \ldots, y^\ell$ be $N^2$-dimensional, $y^j = (y^j_{\nu_1, \nu_2})_{\nu_1, \nu_2 \le N}$. In this case, we set
\[
x^1 \otimes \cdots \otimes x^k \otimes y^1 \otimes \cdots \otimes y^\ell \coloneqq (x_{i_1}^1 \cdots x_{i_k}^k y_{i_{k+1},i_{k+2}}^1 \cdots y_{i_{k+2\ell-1},i_{k+2\ell}}^\ell)_{i_1, \ldots, i_{k+2\ell}\le N},
\]
which we regard as a rectangular $(k + \ell)$-tensor whose first $k$ components are $N$-dimensional and whose last $\ell$ components are $N^2$-dimensional.

\begin{proof}[Proof of Theorem \ref{multlinpol}]
To ease notation, we assume $\abs{\mathcal{X}}=1$ in the sequel. The general case follows in the same way with only minor changes. Recall the fact that for a standard Gaussian $g$ in $\IR^k$ for some $k \in \IN$ and $x \in \IR^k$ we have $\sqrt{p}M^{-1}\abs{x} \le \lVert \langle x, g \rangle \rVert_{L^p} \le M \sqrt{p} \abs{x}$ for all $p \ge 1$ and some universal constant $M > 1$. Combining this and equation \eqref{eqn:momest} we arrive at
\begin{equation}\label{eqn:momestrewr}
    \norm{f - \IE_{\kappa} f}_{L^p(\IP_\kappa)}
        \le K( \norm{\langle \nabla f, G \rangle}_{L^p} + (2N)^{-1/2} \norm{\langle \nabla^2 f, H \rangle}_{L^p}),
\end{equation}
for $K \coloneqq \sqrt{6\theta} M$. Here, $G$ is an $N$-dimensional standard Gaussian and $H$ is an $N^2$-dimensional standard Gaussian such that $G$ and $H$ are independent of each other and of the $\omega_i$, and the $L^p$ norms on the right-hand side are taken with respect to the product measure of $\IP_\kappa$ and the Gaussians.

Note that $\langle \nabla f, G \rangle$ and $\langle \nabla^2 f, H \rangle$ are again multilinear polynomials in the $\omega_i$. Moreover, $\langle \nabla \langle \nabla f, G_1 \rangle, G_2 \rangle = \langle \nabla^2 f, G_1 \otimes G_2 \rangle$ and $\langle \nabla^2 \langle \nabla f, G \rangle, H \rangle = \langle \nabla^3 f, G \otimes H \rangle$. In the last expression, we regard $\nabla^3f$ as a $2$-tensor whose second component is $N^2$-dimensional. Similar relations also hold for the other terms in \eqref{eqn:momestrewr}.

The proof now follows by iterating \eqref{eqn:momestrewr}. For simplicity of presentation, let us consider the case of $d=2$ first. Here, we apply the triangle inequality (in the form $\norm{\langle \nabla f, G \rangle}_{L^p} \le \norm{\langle \IE_\kappa \nabla f, G \rangle}_{L^p} + \norm{\langle \nabla f - \IE_\kappa \nabla f, G \rangle}_{L^p}$ and similarly for $\langle \nabla^2 f, H \rangle$) to \eqref{eqn:momestrewr}. We may then apply \eqref{eqn:momestrewr} to $\langle \nabla f - \IE_\kappa \nabla f, G \rangle$ and $\langle \nabla^2 f - \IE_\kappa \nabla^2 f, H \rangle$ again. This leads to
\begin{align}\label{eqn:akps4.1d=2}
\begin{split}
    &\norm{f - \IE_{\kappa} f}_{L^p(\IP_\kappa)}\\
    \le \ &K \norm{\langle \IE_\kappa \nabla f, G \rangle}_{L^p} + K (2N)^{-1/2} \norm{\langle \IE_\kappa \nabla^2 f, H \rangle}_{L^p}\\
    &+ K^2 \norm{\langle \nabla^2 f, G_1 \otimes G_2 \rangle}_{L^p} + 2 K^2 (2N)^{-1/2} \norm{\langle \IE_\kappa \nabla^3 f, G \otimes H \rangle}_{L^p}\\
    &+ K^2 (2N)^{-1}\norm{\langle \IE_\kappa \nabla^4 f, H_1 \otimes H_2 \rangle}_{L^p}\\
    = \ &K \norm{\langle \IE_\kappa \nabla f, G \rangle}_{L^p} + K (2N)^{-1/2} \norm{\langle \IE_\kappa \nabla^2 f, H \rangle}_{L^p}\\
    &+ K^2 \norm{\langle \IE_\kappa \nabla^2 f, G_1 \otimes G_2 \rangle}_{L^p}
    \end{split}
\end{align}
In the last step, we have used that since $f$ is a multilinear polynomial of degree $2$, its second order derivatives are constant and all derivatives of order larger than $2$ vanish.

Next we use that by \cite{La06}, there are constants $C_k$ depending on $k$ only such that for any (possibly rectangular) $k$-tensor $A$ and any $p \ge 2$,
\begin{equation}\label{eqn:Latala}
    \norm{\langle A, g_1 \otimes \cdots \otimes g_k \rangle}_{L^p} \le C_k \sum_{\mathcal{I} \in P_k} p^{\abs{\mathcal{I}}/2} \norm{A}_\mathcal{I},
\end{equation}
where $g_1, \ldots, g_k$ are standard Gaussians. Applying \eqref{eqn:Latala} to \eqref{eqn:akps4.1d=2}, we obtain for some absolute constant $C$
\begin{align*}
    &\norm{f - \IE_{\kappa} f}_{L^p(\IP_\kappa)}\\
    \le \ &K \norm{\langle \IE_\kappa \nabla f, G \rangle}_{L^p} + K (2N)^{-1/2} \norm{\langle \IE_\kappa \nabla^2 f, H \rangle}_{L^p} + K^2 \norm{\langle \IE_\kappa \nabla^2 f, G_1 \otimes G_2 \rangle}_{L^p}\\
    \le \ &C_1K p^{1/2} \abs{\IE_\kappa \nabla f} + C_1K (2N)^{-1/2} p^{1/2} \norm{\IE_\kappa \nabla^2 f}_\mathrm{HS} + C_2 K^2 p^{1/2} \norm{\IE_\kappa \nabla^2 f}_\mathrm{HS}\\
    &+ C_2 K^2 p \norm{\IE_\kappa \nabla^2 f}_\mathrm{op}\\
    \le \ &C (p^{1/2} \abs{\IE_\kappa \nabla f} + p^{1/2} \norm{\IE_\kappa \nabla^2 f}_\mathrm{HS} + p \norm{\IE_\kappa \nabla^2 f}_\mathrm{op}).
\end{align*}
From here, the assertion follows by standard arguments, cf.\ e.\,g.\ \cite[Proposition 4]{GSS18b}.

Finally, we consider an arbitrary $d \ge 2$ and explain how the proof given above generalizes. First, we apply the triangle inequality to \eqref{eqn:momestrewr} and iterate $d-1$ times. This yields
\begin{equation}\label{eqn:akps4.1}
    \norm{f - \IE_{\kappa} f}_{L^p(\IP_\kappa)}
        \le \psi_d + \sum_{i=1}^{d-1} \psi_i,
\end{equation}
where we have
\begin{align}\label{eqn:akps4.1expl}
    \begin{split}
        \psi_d \coloneqq \sum_{\ell=0}^d \binom{d}{\ell} K^d (2N)^{-\ell/2} \norm{\langle \nabla^{d+\ell}f, G_1 \otimes \cdots \otimes G_{d-\ell} \otimes H_1 \otimes \cdots \otimes H_\ell \rangle}_{L^p},\\
        \psi_i \coloneqq \sum_{\ell=0}^i \binom{i}{\ell} K^i(2N)^{-\ell/2} \norm{\langle \IE_\kappa \nabla^{i+\ell}f, G_1 \otimes \cdots \otimes G_{i-\ell} \otimes H_1 \otimes \cdots \otimes H_\ell \rangle}_{L^p}
    \end{split}
\end{align}
for any $i=1, \ldots, d-1$. As $f$ is a multilinear polynomial of degree $d$, these expressions simplify since the derivatives of order $d$ are constant and all derivatives of higher order vanish. In particular,
\[
\psi_d = K^d \norm{\langle \IE_\kappa \nabla^df, G_1 \otimes \cdots \otimes G_d \rangle}_{L^p}.
\]
Now, as above we apply \eqref{eqn:Latala} to \eqref{eqn:akps4.1} (or rather the $L^p$ norms appearing in \eqref{eqn:akps4.1expl}) to arrive at
\[
\norm{f - \IE_{\kappa} f}_{L^p(\IP_\kappa)} \le C \sum_{k=1}^d \sum_{\mathcal{I} \in P_k} p^{\abs{\mathcal{I}}/2} \norm{\IE_\kappa \nabla^k f}_\mathcal{I}
\]
for some absolute constant $C> 0$ depending on $d$ only. In particular, we use that if we apply \eqref{eqn:Latala} to some $\ell \ge 1$ term in $\psi_i$ in \eqref{eqn:akps4.1expl}, the norms which arise reappear in the norms corresponding to $\ell = 0$ in the $\psi_{i+\ell}$ terms. The proof is concluded by recalling \cite[Proposition 4]{GSS18b} again.
\end{proof}

\begin{proof}[Proof of Proposition \ref{prop:ERDr}]
The proof works by calculating $\norm{\IE_\kappa \nabla^k f}_\mathcal{I}$ for $k = 1, 2, 3$ and applying Theorem \ref{multlinpol}. In the sequel, we use the convention $\omega_{ji} \coloneqq \omega_{ij}$ whenever $j > i$. It is easy to see that for any edge $e = \{i,j\}$, we have
\[
\frac{\partial}{\partial \omega_e} f(\omega) = \sum_{k \in \{1, \ldots, n\}\setminus \{i,j\}} \omega_{ik}\omega_{jk}.
\]
Moreover, the second order derivatives $\partial^2 f/(\partial \omega_{e_1}\partial \omega_{e_2})$ are zero unless $e_1$ and $e_2$ share exactly one vertex, in which case it is $\omega_{ij}$ if $i$ and $j$ are the two vertices distinct from the common one. Finally, the third order derivatives $\partial^3 f/(\partial \omega_{e_1}\partial \omega_{e_2} \partial \omega_{e_3})$ are $1$ if $e_1, e_2, e_3$ form a triangle and zero if not.

Using that
\[
    \IE \omega_{e_1} \cdots \omega_{e_k} = \frac{M(M-1) \cdots (M-k+1)}{N(N-1) \cdots (N-k+1)},
\]
for any $k = 1, \ldots, N$ and any pairwise distinct set of edges $e_1, \ldots, e_k$, we therefore obtain 
\[
    \norm{\IE \nabla f}_{\{1\}} = \sqrt{N} (n-2) \frac{M(M-1)}{N(N-1)} \le n^2 p^2.
\]
Moreover, we have $\IE \nabla^2 f = p (\eins_{\abs{e_1 \cap e_2} = 1})_{e_1,e_2}$, where $\abs{e_1 \cap e_2}$ denotes the number of common vertices of $e_1$ and $e_2$. Therefore, we may use the calculations from the proof of \cite[Proposition 5.5]{AW15}, which yield
\begin{gather*}
    \norm{\IE \nabla^2 f}_{\{1,2\}} \le p n^{3/2},\qquad \norm{\IE \nabla^2 f}_{\{1\},\{
    2\}} \le 2 p n,\\
    \norm{\IE \nabla^3 f}_{\{1,2,3\}} \le n^{3/2},\qquad \norm{\IE \nabla^3 f}_{\{1\},\{2\},\{3\}} \le 2^{3/2},\\
    \norm{\IE \nabla^3 f}_{\{1,2\},\{3\}} = \norm{\IE \nabla^3 f}_{\{1,3\},\{2\}} = \norm{\IE \nabla^3 f}_{\{2,3\},\{1\}} \le \sqrt{2n}.
\end{gather*}
The proof now follows by plugging in.
\end{proof}

The results of Section \ref{sect:swor} follow from the logarithmic Sobolev inequalities established in Section \ref{sect:LSIs} by standard means.

\begin{proof}[Proof of Proposition \ref{theorem:subgtswor}]
    Noting that
    \[
        \mathfrak{h}(f) = \Big( \frac{1}{2} \sum_{i = 1}^n \big(\sup_{\omega_1} f(\omega) - \inf_{\omega_1'} f(\omega_1', \omega_2, \ldots, \omega_n)\big)^2 \Big)^{1/2} \le \Big( \frac{c^2n}{2} \Big)^{1/2},
    \]
    \eqref{eqn:swrboundeddiff1} follows from Proposition \ref{prop:LSIswor} using the arguments from the proof of Proposition \ref{elementaryresults}.
    Now, to prove \eqref{eqn:bdasymmetric} define the symmetric function $g(\omega) \coloneqq \frac{1}{n!} \sum_{\sigma \in S_n} f(\sigma \omega)$, and observe that by exchangeability of the $\omega_i$ we have $\IE_{\kappa,n} f = \IE_{\kappa,n} g$. Moreover,
    \begin{align*}
        \abs{g(\omega) - g(\omega_1', \omega_2, \ldots, \omega_n)} &\le \frac{1}{n!} \sum_{\sigma \in S_n} \abs{f(\sigma \omega) - f(\sigma (\omega_1', \omega_2, \ldots, \omega_n))} \\
        &\le \frac{1}{n!} \sum_{\sigma \in S_n} \sum_{i =1}^n \eins_{\sigma(1) = i} c_i
        \le \frac{1}{n} \sum_{i = 1}^n c_i.
    \end{align*}
    Applying equation \eqref{eqn:swrboundeddiff1} to $g$ and using Jensen's inequality yields
    \[
        \IP_{\kappa,n}(g - \IE_{\kappa,n} g \ge t) \le \exp\Big( - \frac{t^2}{4(1-n/N)n (n^{-1} \sum_i c_i)^2} \Big) \le \exp\Big( - \frac{t^2}{4(1-n/N) \sum_i c_i^2} \Big)
    \]
    as claimed.
\end{proof}

\begin{proof}[Proof of Corollary \ref{corollary:Serfling}]
    This follows immediately from Proposition \ref{theorem:subgtswor}, as $f(\omega) = n^{-1} \sum_{i = 1}^n \omega_i$ is a symmetric function satisfying $\abs{f(\omega) - f(\omega_1', \omega_2, \ldots, \omega_n)} \le \abs{\mathcal{X}}/n$.
\end{proof}

\begin{proof}[Proof of Corollary \ref{corollary:subgconj}]
This is a consequence of Proposition \ref{theorem:subgtswor}, as for any $\omega \in \Omega_{\kappa,n}$ and $\omega_1'$ we have by the reverse triangle inequality
\begin{align*}
    \abs{f(\omega) - f(\omega_1', \omega_2, \ldots, \omega_n)} \le n^{-1} \sup_{t \in \IR} \abs{\eins_{(-\infty, t]}(\omega_1) - \eins_{(-\infty,t]}(\omega_1')} \le n^{-1}.
\end{align*}
\end{proof}

To prove Talagrand's convex distance inequality on the multislice, we follow the approach by Boucheron, Lugosi and Massart \cite{BLM03}, see also \cite[Proposition 1.9]{SS19}. A key step in the proof is the following lemma.

\begin{lemma}\label{lemfTcdeMS}
Let $f: \Omega_\kappa \to \IR$ be a non-negative function such that
\begin{enumerate}
    \item $\Gamma^+(f)^2 \le f$,
    \item $\abs{f(\omega) - f(\tau_{ij}\omega)} \le 1$ for all $\omega, i,j$.
\end{enumerate}
Then for all $t \in [0, \IE_\kappa f]$ we have
\[
\IP_\kappa(\IE_\kappa f - f \ge t) \le \exp\Big( -\frac{t^2}{32\IE_\kappa f} \Big).
\]
Especially we have
\[
    \IP_\kappa(f = 0)\exp\Big( \frac{\IE_\kappa f}{32} \Big) \le 1.
\]
In particular, this holds for $f(\omega) = \frac{1}{4} d_T(\omega, A)^2$, where $A \subset S_n$ is any set.
\end{lemma}

We defer the proof of Lemma \ref{lemfTcdeMS} until the end of the section and first show how to apply it to prove Talagrand's convex distance inequality.

\begin{proof}[Proof of Proposition \ref{TcdeMS}]
The difference operator $\Gamma^+$ clearly satisfies $\Gamma^+(g^2) \le 2g \Gamma^+(g)$ for all positive functions $g$, as well as a $\Gamma^+-\mathrm{mLSI}(8)$. Moreover, as seen in the proof of Lemma \ref{lemfTcdeMS}, we have $\Gamma^+(d_T(\cdot, A)) \le 1$. Thus, by \cite[(3.6)]{SS19} it holds for $\lambda \in [0,1/16)$
\[
    \IP_\kappa(A) \IE_\kappa \exp\Big( \lambda d_T(\cdot,A)^2 \Big) \le \IP_\kappa(A) \exp\Big( \frac{\lambda}{1-16\lambda} \IE_\kappa d_T(\cdot,A)^2 \Big).
\]
Furthermore, Lemma \ref{lemfTcdeMS} shows that
\[
    \IP_\kappa(A) \exp\Big( \frac{\IE_\kappa d_T(\cdot,A)^2}{128} \Big) \le 1.
\]
So, for $\lambda = 1/144$ we have
\[
\IP_\kappa(A) \IE_\kappa \exp\Big( \frac{d_T(\cdot,A)^2}{144} \Big) \le \IP_\kappa(A) \exp\Big( \frac{1}{128} \IE_\kappa d_T(\cdot,A)^2 \Big) \le 1.
\]
\end{proof}

\begin{proof}[Proofs of Corollaries \ref{cor1} and \ref{corollary:convSymmLip}]
These corollaries follow in exactly the same way as the proof of \cite[Theorem 3]{Tal88}. The only difference is to note that for any $x, y \in \{ f \le \mathrm{med}(f) \}$ such that $f(x) \ge \mathrm{med}(f) + t$ we have
\[
t \le \mathrm{med}(f) + t - f(y) \le f(x) - f(y) \le L \abs{x-y} \le L \abs{\mathcal{X}} \sup_{\alpha \in \IR^n : \abs{\alpha = 1}} \sum_{i= 1}^n \alpha_i \eins_{x_i \neq y_i},
\]
so that
\[
f(x) \ge \mathrm{med}(f) + t \Rightarrow d_T(x, A) \ge t/(\abs{\mathcal{X}} L).
\]
\end{proof}

\begin{proof}[Proof of Corollary \ref{cor:largestEV}]
Since $\lambda_\mathrm{max} = \norm{X}_\mathrm{op}$, it is clear by triangular inequality that $\lambda_\mathrm{max}$ is a convex function of the $X_{ij}$, $i \le j$. Moreover, due to Lidskii's inequality, $\lambda_\mathrm{max}$ is $1$-Lipschitz. It therefore remains to apply Corollary \ref{cor1}.
\end{proof}

\begin{proof}[Proof of Lemma \ref{lemfTcdeMS}]
Rewriting \cite[Lemma 1]{Sal20}, for any positive function $g$ it holds
\begin{align*}
\mathrm{Ent}_\kappa(g)
&\le \frac{1}{2N} \sum_{i,j} \IE_\kappa (g(\tau_{ij}\omega) - g(\omega))(\log g(\tau_{ij}\omega) - \log g(\omega)) \\
&= \frac{1}{N} \sum_{i,j} \IE_\kappa (g(\tau_{ij}\omega) - g(\omega)) (\log g(\tau_{ij}\omega) - \log g(\omega))_+.
\end{align*}
Using this, we obtain for any $\lambda \in [0,1]$
\begin{align*}
    \mathrm{Ent}_\kappa(e^{-\lambda f}) &\le \frac{\lambda}{N} \IE_\kappa \sum_{i,j} (f(\omega) - f(\tau_{ij}\omega))_+ \big( \exp(-\lambda f(\tau_{ij}\omega)) - \exp(-\lambda f(\omega)) \big) \\
    &= \frac{\lambda}{N} \IE_\kappa \sum_{i,j} (f(\omega) - f(\tau_{ij} \omega))_+ (\exp(\lambda(f(\omega) - f(\tau_{ij}\omega)))-1) e^{-\lambda f(\omega)} \\
    &\le \frac{\lambda}{N} \IE_\kappa \sum_{i,j} (f(\omega) - f(\tau_{ij}\omega))_+ \Psi(\lambda(f(\omega) - f(\tau_{ij}\omega))) e^{-\lambda f(\omega)},
\end{align*}
where $\Psi(x) \coloneqq e^x - 1$. By a Taylor expansion it can easily be seen that $\Psi(x) \le 2x$ for all $x \in [0,1]$, so that (recall that by $(2)$ we have $f(\omega) - f(\tau_{ij}\omega) \le 1$, and $f(\omega) - f(\tau_{ij}\omega) \ge 0$ due to the positive part)
\begin{align*}
\mathrm{Ent}_\kappa(e^{-\lambda f}) &\le \frac{2\lambda^2}{N}  \IE_\kappa \sum_{i,j} (f(\omega) - f(\tau_{ij}\omega))_+^2 e^{-\lambda f(\omega)}\\ &= 8\lambda^2 \IE_\kappa \Gamma^+(f)^2 e^{-\lambda f} \le 8\lambda^2 \IE_\kappa f e^{-\lambda f}.
\end{align*}
The covariance of $fe^{-\lambda f}$ is non-positive (i.\,e.\ $\IE f e^{-\lambda f} \le \IE f \IE e^{-\lambda f}$), which yields
\[
\mathrm{Ent}_\kappa(e^{-\lambda f}) \le 8\lambda^2 \IE_\kappa f \IE_\kappa e^{-\lambda f}.
\]
In other terms, if we set $h(\lambda) \coloneqq \IE_\kappa e^{-\lambda f}$, we have
\[
\Big( \frac{\log h(\lambda)}{\lambda} \Big)' \le 8\IE_\kappa f,
\]
which by the fundamental theorem of calculus implies for all $\lambda \in [0,1]$
\[
\IE_\kappa \exp\Big( \lambda(\IE_\kappa f - f) \Big) \le \exp\Big( 8\lambda^2 \IE_\kappa f \Big).
\]
So, for any $t \in [0, \IE_\kappa f]$, by Markov's inequality and setting $\lambda = \frac{t}{16 \IE_\kappa f}$
\[
\IP_\kappa(\IE_\kappa f - f \ge t) \le \exp\Big(-\lambda t + 8\lambda^2 \IE_\kappa f\Big) = \exp\Big( - \frac{t^2}{32\IE_\kappa f} \Big).
\]
The second part follows by nonnegativity and $t = \IE_\kappa f$.

It remains to check that $f(\omega) = \frac{1}{4} d_T(\omega, A)^2$ satisfies the two conditions of this lemma. To this end, we first show that $\Gamma^+(d_T(\cdot,A))^2 \le 1$. Writing $g(\omega) \coloneqq d_T(\omega, A)$, it is well known (see \cite{BLM03}) that by Sion's minimax theorem, we have
\begin{equation}\label{Sion}
g(\omega) = \inf_{\nu \in \mathcal{M}(A)} \sup_{\alpha \in \IR^N : \abs{\alpha} = 1} \sum_{k = 1}^N \alpha_k \nu(\omega' : \omega'_k \neq \omega_k),
\end{equation}
where $\mathcal{M}(A)$ is the set of all probability measures on $A$. To estimate $\Gamma^+(g)^2(\omega)$, one has to compare $g(\omega)$ and $g(\tau_{ij}\omega)$. To this end, for any $\omega \in \Omega_\kappa$ fixed, let $\tilde{\alpha}, \tilde{\nu}$ be parameters for which the value $g(\omega)$ is attained, and let $\hat{\nu} = \hat{\nu}_{ij}$ be a minimizer of $\inf_{\nu \in \mathcal{M}(A)} \sum_{k = 1}^N \tilde{\alpha}_k \nu(\omega' : \omega'_k \neq (\tau_{ij}\omega)_k)$. This leads to
\begin{align*}
    \Gamma^+ (g)(\omega)^2 &\le \frac{1}{4N} \sum_{i,j = 1}^N \Big( \sum_{k = 1}^N \tilde{\alpha}_k (\hat{\nu}(\omega'_k \neq \omega_k) - \hat{\nu}(\omega'_k \neq (\tau_{ij}\omega)_k))\Big)_+^2 \\
    &\le \frac{1}{2N} \sum_{i,j = 1}^N (\tilde{\alpha}_i^2 + \tilde{\alpha}_j^2)
    \le 1.
\end{align*}
Using this as well as $\Gamma^+(g^2) \le 2g\Gamma^+(g)$ for all positive functions $g$, we have 
\[
\Gamma^+(f)^2 = \frac{1}{16} \Gamma^+(d_T(\cdot, A)^2)^2 \le \frac{1}{4} d_T(\cdot,A)^2 \Gamma^+(d_T(\cdot,A))^2 \le f.
\]

To show the second property, we proceed similarly to \cite[Proof of Lemma 1]{BLM09}. By \eqref{Sion} and the Cauchy--Schwarz inequality, we have
\[
f(\omega) = \frac{1}{4} \inf_{\nu \in \mathcal{M}(A)} \sum_{k = 1}^N \nu(\omega' : \omega'_k \neq \omega_k)^2.
\]
Assuming without loss of generality that $f(\omega) \ge f(\tau_{ij}\omega)$, choose $\hat{\nu} = \hat{\nu}_{ij} \in \mathcal{M}(A)$ such that the value of $f(\tau_{ij}\omega)$ is attained. It follows that
\begin{equation*}
    f(\omega) - f(\tau_{ij}\omega) \le \frac{1}{4}\sum_{k=1}^N \hat{\nu}(\omega'_k \neq \omega_k)^2 - \hat{\nu} (\omega'_k \neq (\tau_{ij}\omega)_k)^2 \le \frac{2}{4},
\end{equation*}
which finishes the proof.
\end{proof}

\begin{proof}[Proof of Theorem \ref{thm:convswor}]
Since $A$ is a symmetric set, $\omega \mapsto d_T(\omega,A)$ is a symmetric function, which follows by the definition
\[
d_T(\omega, A) = \sup_{\alpha \in \IR^n : \abs{\alpha} = 1} \inf_{\omega' \in A} \sum_{i = 1}^n \abs{\alpha_i} \eins_{\omega_i \neq \omega_i'}.
\]
As in \eqref{Sion}, we may use Sion's minimax theorem to rewrite $d_T$ as
\[
    d_T(\omega,A) = \inf_{\nu \in \mathcal{M}(A)} \sup_{\alpha \in \IR^n : \abs{\alpha} = 1} \sum_{k = 1}^n \alpha_k \nu( \omega' : \omega'_k \neq \omega_k ).
\]
As in the proof of Proposition \ref{TcdeMS}, let $\tilde{\nu}, \tilde{\alpha}$ be the parameters for which the value $d_T(\omega,A)$ is attained, and let $\hat{\nu}$, $\hat{\omega}_i'$ be minimizers of $\inf_{\omega_i'} \inf_{\nu \in \mathcal{M}(A)} \sum_{j = 1}^n \tilde{\alpha}_k \nu( \eta : \eta_k \neq (\omega_{i^c}, \omega_i')_k)$. We then have
\begin{align*}
\mathfrak{h}^+(d_T(\omega, A))^2 &= \frac{1}{2} \sum_{i = 1}^n \big( d_T(\omega,A) - \inf_{\omega_i'} d_T((\omega_{i^c}, \omega_i'), A \big)_+^2 \\
&\le \frac{1}{2} \sum_{i = 1}^n \big( \sum_{k = 1}^n \tilde{\alpha}_k \hat{\nu}( \eta : \eta_k \neq \omega_k) - \sum_{k = 1}^n \tilde{\alpha}_k \hat{\nu}( \eta : \eta_k \neq (\omega_{i^c}, \hat{\omega}_i')_k) \big)_+^2 \\
&\le \frac{1}{2} \sum_{i = 1}^n \tilde{\alpha}_i^2 = \frac{1}{2}
\end{align*}
Recall that by Proposition \ref{prop:LSIswor}, $\IP_{\kappa,n}$ satisfies an $\mathfrak{h}^+-\mathrm{LSI}(8(1-\frac{n}{N}))$ on the set of all symmetric functions. As a consequence, using \eqref{eqn:SS19-1.2} again, we obtain the sub-Gaussian estimate
\begin{align*}
    \IP_{\kappa,n}\big( d_T(\cdot, A) - \IE_{\kappa,n} d_T(\cdot, A) \ge t \big) \le \exp\Big( - \frac{t^2}{8(1-n/N)} \Big).
\end{align*}

In the next step, we observe that by the Poincaré inequality we have
\[
\mathrm{Var}(d_T(\cdot,A)) \le 8 (1-n/N)\IE_{\kappa,n} \mathfrak{h}^+(d_T(\cdot,A))^2 \le 4(1-n/N).
\]
Hence, Chebyshev's inequality leads to
\[
(\IE_{\kappa,n} d_T(\cdot, A))^2 \IP_{\kappa, n}\big(d_T(\cdot,A) - \IE_{\kappa,n} d_T(\cdot, A) \le - \IE_{\kappa,n} d_T(\cdot, A)\big) \le 4(1-n/N).
\]
Using that $\IP_{\kappa,n}(A) \ge 1/2$, we therefore have $\IE_{\kappa, n} d_T(\cdot, A) \le \sqrt{8(1-n/N)}$. Finally, since $(t-a)^2 \ge t^2/2 - a^2$ for any $a \in \IR$ we obtain for $t \ge \sqrt{8(1-n/N)}$
\[
\IP_{\kappa,n}(d_T(\cdot, A) \ge t) \le \exp\Big( - \frac{(t-\sqrt{8(1-n/N)})^2}{8(1-n/N)} \Big) \le \exp\Big( - \frac{t^2}{16(1-n/N)} + 1 \Big).
\]
For $t \le \sqrt{8(1-n/N)}$ the inequality holds trivially, which finishes the proof.
\end{proof}

\printbibliography

 \end{document}